\documentclass[12pt,reqno]{amsart}
\usepackage{amsmath}
\usepackage{amsthm}
\usepackage{amsfonts}
\usepackage{amssymb}
\usepackage{mathrsfs}
\usepackage[margin=1.35in, top =1.25in, bottom = 1.25in]{geometry}
\usepackage{setspace}
\usepackage{tikz}
\usetikzlibrary{matrix, arrows}
\usepackage{tikz-cd}
\usepackage[cal=euler]{mathalfa}
\usepackage{enumitem}
\usepackage[colorlinks, breaklinks]{hyperref}
\hypersetup{linkcolor=black,citecolor=blue,filecolor=black,urlcolor=black} 
\usepackage{pdflscape}
\usepackage{array}
\usepackage{adjustbox}

\DeclareMathOperator{\ass}{Ass}
\DeclareMathOperator{\cls}{cls}				
\newcommand{\dd}{\partial}
\renewcommand{\epsilon}{\varepsilon}
\DeclareMathOperator{\hgt}{ht}  			
\DeclareMathOperator{\init}{in}				
\newcommand{\iso}{\cong}				
\newcommand{\kk}{k}					
\newcommand{\longto}{\longrightarrow}
\DeclareMathOperator{\pd}{pd}  			
\renewcommand{\phi}{\varphi}
\newcommand{\PP}{\mathbb{P}}			
\newcommand{\QQ}{\mathbb{Q}}
\DeclareMathOperator{\reg}{reg}			
\DeclareMathOperator{\Span}{Span}
\DeclareMathOperator{\syz}{Syz}
\newcommand{\tensor}{\otimes}
\DeclareMathOperator{\tor}{Tor}

\newtheorem{thm}{Theorem}[section]
\newtheorem{lemma}[thm]{Lemma}
\newtheorem{prop}[thm]{Proposition}
\newtheorem{cor}[thm]{Corollary}

\newtheorem*{main-thm}{Main Theorem}

\theoremstyle{definition}

\newtheorem{example}[thm]{Example}
\newtheorem{rmk}[thm]{Remark}
\newtheorem{question}[thm]{Question}
\newtheorem*{notation}{Notation}

\numberwithin{equation}{section}
\numberwithin{figure}{section}

\title{Koszul almost complete intersections}
\author{Matthew Mastroeni}
\address{Department of Mathematics, University of Illinois, Urbana, IL 61801}
\email{mastroe2@illinois.edu}
\date{}

\begin{document}
\maketitle

\begin{abstract}
Let $R = S/I$ be a quotient of a standard graded polynomial ring $S$ by an ideal $I$ generated by quadrics.  If $R$ is Koszul, a question of Avramov, Conca, and Iyengar asks whether the Betti numbers of $R$ over $S$ can be bounded above by binomial coefficients on the minimal number of generators of $I$.  Motivated by previous results for Koszul algebras defined by three quadrics, we give a complete classification of the structure of Koszul almost complete intersections and, in the process, give an affirmative answer to the above question for all such rings.
\end{abstract}

\begin{spacing}{1.15}
\section{Introduction}

Let $\kk$ be a field, $S$ be a standard graded polynomial ring over $\kk$, $I \subseteq S$ be a graded ideal, and $R = S/I$.  We say that $R$ is a \emph{Koszul algebra}\index{Koszul!algebra} if $\kk \iso R/R_+$ has a linear free resolution over $R$.  Many rings arising from algebraic geometry are Koszul, including the coordinate rings of Grassmannians \cite{Kempf}, sets of $r \leq 2n$ points in general position in $\PP^n$ \cite{points:in:projective:space}, and canonical embeddings of smooth curves under mild restrictions \cite{canonical:rings:of:curves}, as well as all suitably high Veronese subrings of any standard graded algebra \cite{high:Veronese:subrings:are:Koszul}.  However, the simplest examples of Koszul algebras, due to Fr\"oberg \cite{quadratic:monomial:ideals:are:Koszul}, are quotients by quadratic monomial ideals, and a guiding heuristic in the study of Koszul algebras has been that any reasonable property of algebras defined by quadratic monomial ideals should also hold for Koszul algebras; for example, see \cite{free:resolutions:over:Koszul:algebras}, \cite{Koszul:algebras:and:their:syzygies}, \cite{subadditivity:of:Betti:numbers}.  Among such properties, considering the Taylor resolution for an algebra defined by a quadratic monomial ideal leads to the following question about the Betti numbers of a Koszul algebra.

\begin{question}[{\cite[6.5]{free:resolutions:over:Koszul:algebras}}] \label{Betti:number:bound:for:Koszul:algebras}
If $R$ is Koszul and $I$ is minimally generated by $g$ elements, does the following inequality hold for all $i$?
\[ \beta^S_i(R) \leq \binom{g}{i} \]
In particular, is $\pd_S R \leq g$?
\end{question}

The above questions are known to have affirmative answers when $R$ is LG-quadratic (see next section) and for arbitrary Koszul algebras when $g \leq 3$ by \cite[4.5]{Koszul:algebras:defined:by:3:quadrics}.  Recall that $R$ or $I$ is called an \emph{almost complete intersection} if $I$ is minimally generated by $\hgt I+1$ elements.  The motivation for studying Koszul almost complete intersections comes from the fact that the above question is easily seen to have an affirmative answer when $I$ is a complete intersection or has height one so that the interesting case for Koszul algebras defined by three quadrics is precisely when $I$ is an almost complete intersection.  Our main results (\ref{Koszul:algebras:with:one:linear:syzygy}, \ref{Koszul:ACI's:with:two:linear:syzygies}, \ref{Koszul:ACI's:are:LG-quadratic}, \ref{linear:syzygies:of:quadric:ACIs}) show that Question \ref{Betti:number:bound:for:Koszul:algebras} has an affirmative answer for Koszul almost complete intersections generated by any number of quadrics; they are summarized in the theorem below.

\begin{main-thm}
Let $R = S/I$ be a Koszul almost complete intersection with $I$ minimally generated by $g+1$ quadrics for some $g \geq 1$.  Then $\beta_{2,3}^S(R) \leq 2$, and: \vspace{2ex}
\begin{enumerate}[label = \textnormal{(\alph*)}]
\item If $\beta_{2,3}^S(R) = 1$, then $I = (xz, zw, q_3, \dots, q_{g+1})$ for some linear forms $x$, $z$, and $w$ and some regular sequence of quadrics $q_3, \dots, q_{g+1}$ on $S/(xz,zw)$.
\vspace{2ex}
\item If $\beta_{2,3}^S(R) = 2$, then $I = I_2(M) + (q_4, \dots, q_{g+1})$ for some $3 \times 2$ matrix of linear forms $M$ with $\hgt I_2(M) = 2$ and some regular sequence of quadrics $q_4, \dots, q_{g+1}$ on $S/I_2(M)$.
\end{enumerate}
\vspace{1ex}
Furthermore, $R$ is LG-quadratic and, therefore, satisfies $\beta_i^S(R) \leq \binom{g+1}{i}$ for all $i$.
\end{main-thm}

The division of the rest of the paper is as follows.  We recount various properties and examples of Koszul algebras and their Betti tables in \S \ref{basic:properties:and:examples} which will be important in the sequel.  In \S \ref{Koszul:ACI's}, we determine the structure of Koszul almost complete intersections with either one or two linear second syzygies.  We then complete the classification of Koszul almost complete intersections in \S \ref{linear:syzygies:of:quadratic:ACI's} by showing that every quadratic almost complete intersection has at most two linear second syzygies. 

\begin{notation} 
Throughout the remainder of the paper, the following notation will be in force unless specifically stated otherwise.  Let $\kk$ be a fixed ground field of arbitrary characteristic, $S$ be a standard graded polynomial ring over $\kk$,  $I \subseteq S$ be a proper graded ideal, and $R = S/I$.  Recall that the ideal $I$ is called \emph{nondegenerate}\index{nondegenerate} if it does not contain any linear forms.  We can always reduce to a presentation for $R$ with $I$ nondegenerate by killing a basis for the linear forms contained in $I$, and we will assume that this is the case throughout.  We denote the irrelevant ideal of $R$ by $R_+ = \bigoplus_{n \geq 1} R_n$.
\end{notation}

\section{Koszul Algebras and Their Betti Tables}
\label{basic:properties:and:examples}

If $R$ is a Koszul algebra, it is well-known that its defining ideal $I$ must be generated by quadrics, but not every ideal generated by quadrics defines a Koszul algebra.  We have already noted in the introduction that every quadratic monomial ideal defines a Koszul algebra.  More generally, we say that $R$ or $I$ is \emph{G-quadratic}\index{G-quadratic} if, after a suitable linear change of coordinates $\phi: S \to S$, the ideal $\phi(I)$ has a Gr\"obner basis consisting of quadrics.  We also say that $R$ or $I$ is \emph{LG-quadratic} if $R$ is a quotient of a G-quadratic algebra $A$ by an $A$-sequence of linear forms.  Every G-quadratic algebra is Koszul by upper semicontinuity of the Betti numbers; see \cite[3.13]{upper:semicontinuity}.  It then follows from Proposition \ref{passing:Koszulness:to:and:from:quotients} below that every LG-quadratic algebra is also Koszul. In particular, every complete intersection generated by quadrics is LG-quadratic by an argument due to Caviglia.  Indeed, if $R = S/(q_1, \dots, q_g)$ where $q_1, \dots, q_g$ is a regular sequence of quadrics, we can take $A = S[y_1,\dots, y_g]/(y_1^2+q_1, \dots, y^2_g + q_g)$ so that $A/(y_1, \dots, y_g) \iso R$.  By choosing a monomial order in which the $y_i$ are greater than every monomial in the variables of $S$, it follows from \cite[15.15]{Eisenbud} that $A$ is G-quadratic and that the $y_i^2+q_i$ form a regular sequence so that $\hgt_A (y_1, \dots, y_g) = \dim A - \dim R = \dim S - \dim R = g$ and $y_1, \dots, y_g$ is an $A$-sequence.  In summary, we have the following implications.

\vspace{2ex}
\begin{center}
\begin{tikzcd}
\text{G-quadratic} \rar[Rightarrow] & \text{LG-quadratic} \rar[Rightarrow] & \text{Koszul}
\\
& \text{Quadratic CI's} \uar[Rightarrow]
\end{tikzcd}
\end{center}
 \vspace{2ex}

Each of the above implications is strict.  Clearly, any quadratic monomial ideal which is not a complete intersection, such as $(xy, xz, xw) \subseteq \kk[x,y,z,w]$, is LG-quadratic.  In \cite[1.14]{Koszul:algebras:and:their:syzygies}, it is observed that $(x^2+yz, y^2+xz, z^2+xy) \subseteq \QQ[x,y,z]$ is an Artinian quadratic complete intersection which cannot be G-quadratic since it does not contain the square of a linear form.  We will see an example of a Koszul algebra which is not LG-quadratic below.

\begin{rmk}
If $R = S/I$ is G-quadratic and $J$ is a quadratic initial ideal of $I$, then $\beta_1^S(R) = \beta_1^S(S/J)$ since $I$ is generated by quadrics and $R$ and $S/J$ have the same Hilbert function.  Consequently, $\beta_i(R) \leq \beta_i(S/J) \leq \binom{g}{i}$ for all $i$ by upper semicontinuity of the Betti numbers and the Taylor resolution for $S/J$.  Since killing a regular sequence of linear forms does not affect the Betti numbers of $R$, it follows that Question \ref{Betti:number:bound:for:Koszul:algebras} has an affirmative answer for every LG-quadratic algebra.
\end{rmk}

We will be specifically interested in the \emph{graded Betti numbers} of a Koszul algebra $R$, which are defined by $\beta_{i,j}^S(R) = \dim_\kk \tor_i^S(\kk, R)_j$ and related to the usual Betti numbers by $\beta_i^S(R) = \sum_j \beta_{i,j}^S(R)$.  This information is usually organized into a table, called the Betti table of $R$; see below for an example.  As we have already pointed out in the the introduction, quadratic monomial ideals serve as a useful benchmark in the study of Koszul algebras.  In particular, we note that the square-free quadratic monomial ideals are precisely edge ideals.  Recall that, if $G$ is a graph with vertex set $[n] = \{1, \dots, n\}$, the \emph{edge ideal}\index{edge ideal} of $G$ is the ideal of the polynomial ring $S = \kk[x_1, \dots, x_n]$ defined by $I_G = (x_ix_j \mid ij \in E(G))$.  We note that every quadratic monomial ideal can be obtained as the image of an edge ideal modulo a regular sequence of linear forms via polarization \cite[4.2.16]{Bruns:Herzog}, and hence, studying the Betti tables of all quadratic monomial ideals with $g$ generators is equivalent to studying the Betti tables of edge ideals of graphs with $g$ edges, which are reasonably simple to enumerate in practice for small values of $g$.  

In fact, a byproduct of the proof in \cite{Koszul:algebras:defined:by:3:quadrics} that every Koszul algebra defined by $g \leq 3$ quadrics satisfies Question \ref{Betti:number:bound:for:Koszul:algebras} is that every such algebra has the Betti table of some edge ideal.  This suggests taking the Betti tables of edge ideals as our guide for exploring possible patterns in Betti tables of Koszul algebras with more generators.  One can then easily compute that, for various values of $g$, there are only two possible Betti tables for almost complete intersection edge ideals with $g$ generators, one with a single linear syzygy and another with two.  The purpose of this paper is to show that this pattern holds more generally for all Koszul almost complete intersections.  However, the following example shows that the mantra that Koszul algebras are similar to quotients by quadratic monomial ideals must be taken with a grain of salt.

\begin{example}[{\cite[3.8]{Koszul:algebras:and:their:syzygies}}] \label{Koszul:but:not:LG-quadratic}
The ring $R = \kk[x,y,z,w]/(xy, xw, (x-y)z, z^2, x^2+zw)$ is Koszul by a filtration argument.  The minimal free resolution of $R$ over $S = \kk[x,y,z,w]$ can be computed via iterated mapping cones using the fact that $((xy, xw, z^2): (x-y)z) = (xy, xw, z)$ and $((xy, xw, z^2, (x-y)z) : x^2+zw) = (xw, y, z)$.  This yields the following Betti table for $R$, where the entry in column $i$ and row $j$ is $\beta_{i,i+j}^S(R)$ and zero entries are represented by ``$-$'' for readability.
 \vspace{1ex}
\begin{center}
   \begin{tabular}{c|cccccccc}
  & 0 & 1 & 2 & 3 & 4 \\ 
\hline 
0 & 1 & -- & --  & -- & -- \\ 
1 & -- & 5 & 4  & -- & -- \\
2 & -- & -- & 4 & 6 & 2 \\
\end{tabular}
\end{center}
 \vspace{1ex}
From the above Betti table, we see that the Hilbert series of $R$ is 
 \vspace{1ex}
\[ H_R(t) = \frac{1+2t-2t^2-2t^3+2t^4}{(1-t)^2} \;.  \vspace{1ex}\]
If $R$ were LG-quadratic, then the numerator of the Hilbert series must be the $h$-polynomial of a 5-generated edge ideal, since killing a regular sequence of linear forms and passing to a quadratic initial ideal do not change either the $h$-polynomial or the minimal number of generators of the defining ideal.  As noted above, we can easily compute that the $h$-polynomial of $R$ does not belong to any 5-generated edge ideal.  Hence, $R$ is not LG-quadratic, and in particular, the Betti table of $R$ is not the Betti table of any edge ideal.  This points to unexpected complications in trying to answer Question \ref{Betti:number:bound:for:Koszul:algebras} for Koszul algebras defined by $g \geq 5$ quadrics.
\end{example}

\begin{rmk}
We can compute the Betti tables of 5-generated edge ideals in the above example over a field of any characteristic by a result of Katzman, \cite[4.1]{characteristic:dependence:of:Betti:tables}.  However, Katzman also shows that the Betti tables of edge ideals do depend on the characteristic of the ground field in general. 
\end{rmk}

In the remainder of this section, we collect a few results about Koszul algebras that will be useful in the sequel.  The first of these results states how the Koszul property can be passed to and from quotient rings.

\begin{prop}[{\cite[\S 3.1, 2]{Koszul:algebras:and:regularity}}] \label{passing:Koszulness:to:and:from:quotients}
Let $S$ be a standard graded $\kk$-algebra and $R$ be a quotient ring of $S$.
\begin{enumerate}[label = \textnormal{(\alph*)}]
\item If $S$ is Koszul and $\reg_S(R) \leq 1$, then $R$ is Koszul.
\item If $R$ is Koszul and $\reg_S(R)$ is finite, then $S$ is Koszul.
\end{enumerate}
\end{prop}

Compared with general quadratic algebras, the Betti tables of Koszul algebras are much more restricted.  The following result, discovered in \cite{Backelin} and \cite[4]{Kempf}, says that the Betti tables of Koszul algebras have nonzero entries only on or above the diagonal; see \cite[2.10]{Koszul:algebras:and:their:syzygies} for an easier argument using regularity.

\begin{lemma} \label{Koszul:algebras:have:subdiagonal:Betti:table}
If $R = S/I$ is a Koszul algebra, then $\beta_{i,j}^S(R) = 0$ for all $i$ and $j > 2i$.
\end{lemma}

In addition, the extremal portions of the Betti table of a Koszul algebra $R$, namely the diagonal entries and the linear strand of $I$, satisfy bounds similar to those in Question \ref{Betti:number:bound:for:Koszul:algebras}.

\begin{prop}[{\cite[3.4, 4.2]{Koszul:algebras:defined:by:3:quadrics}}] \label{Koszul:Betti:table:constraints}
Suppose that $R = S/I$ is Koszul and that $I$ is minimally generated by $g$ elements.  Then:
\begin{enumerate}[label = \textnormal{(\alph*)}]
\item $\beta_{i,i+1}^S(R) \leq \binom{g}{i}$ for $2 \leq i \leq g$, and if equality holds for $i = 2$, then $I$ has height one and a linear resolution of length $g$.
\item $\beta_{i,2i}^S(R) \leq \binom{g}{i}$ for $2 \leq i \leq g$, and if equality holds for some $i$, then $I$ is a complete intersection.
\end{enumerate}
\end{prop}

\begin{cor} \label{nonCI:Koszul:algebras:have:linear:syzygies}
If $R = S/I$ is a Koszul algebra which is not a complete intersection, then $I$ has a linear syzygy.
\end{cor}

\begin{proof}
Suppose that $I$ is minimally generated by $q_1, \dots, q_g$.  If the Koszul syzygies on the $q_i$ are all minimal generators of $\syz_1^S(I)$, then $\beta_{2,4}^S(R) \geq \binom{g}{2}$ contradicting the preceding proposition.  Hence, some $\kk$-linear combination of the Koszul syzygies is not minimal, and therefore, it is an $S$-linear combination of linear syzygies.
\end{proof}

Lastly, we will need a fact about the syzygies of a Koszul algebra which will be clear to experts, but for completeness, we give a quick proof.  The proof relies on the product structure on $\tor_*^S(R, \kk)$.  We briefly recall how this product is defined and refer the reader to \cite{infinite:free:resolutions} for further details.  To simplify notation, all tensor products below are over $S$.  

If $F_\bullet$ denotes the minimal free resolution of $R$ over $S$, we have the K\"unneth map
\[ \tor_i^S(R,\kk) \tensor \tor_j^S(R,\kk) = H_i(F_\bullet \tensor \kk) \tensor H_j(F_\bullet \tensor \kk) \stackrel{\kappa}{\longto}  H_{i+j}((F_\bullet \tensor \kk)  \tensor (F_\bullet \tensor \kk)) \]
sending $\cls(v) \tensor \cls(w) \mapsto \cls(v \tensor w)$. Denoting by $\mu^\kk: \kk \tensor \kk \to \kk$ and $\mu^R: R \tensor R \to R$ the respective product maps, we have a chain map $\mu^F: F_\bullet \tensor F_\bullet \to F_\bullet$ lifting $\mu^R$.  The product structure on $\tor_*^S(R, \kk)$ is the composition of the K\"unneth map with the map induced on homology by the chain map
\[ F_\bullet \tensor \kk \tensor F_\bullet \tensor \kk \iso F_\bullet \tensor F_\bullet \tensor \kk \tensor \kk \stackrel{\mu^F \tensor \mu^\kk}{\longto} F_\bullet \tensor \kk \; .\]

\begin{prop} \label{Koszul:syzygies:span}
If $R = S/I$ is a Koszul algebra, then $\syz_1^S(I)$ is minimally generated by linear syzygies and Koszul syzygies.
\end{prop}

\begin{proof}
By Lemma \ref{Koszul:algebras:have:subdiagonal:Betti:table}, we know that $\syz_1^S(I)$ is minimally generated by linear and quadratic syzygies.  We may assume that $\beta_{2,4}^S(R) \neq 0$ or else the conclusion holds trivially.  In that case, it follows from \cite[3.1]{free:resolutions:over:Koszul:algebras} that $\tor_2^S(R,\kk)_4 = (\tor_1^S(R,\kk)_2)^2$, so it suffices to note that the products of the generators of $\tor_1^S(R, \kk)_2$ correspond to the Koszul syzygies on a minimal set of generators $q_1,\dots, q_g$ for $I$.  If $e_1, \dots, e_g$ denotes the standard basis of $S(-2)^g = F_1$ such that $\dd(e_i) = q_i$ for each $i$, we can choose $\mu^F$ so that $\mu^F(e_i \tensor 1) = \mu^F(1 \tensor e_i) = e_i$ for all $i$.  Since the $e_i \tensor 1$ span $\tor_1^S(R, \kk)_2$, it follows that the $\mu^F(e_i \tensor e_j) \tensor 1$ span $\tor_2^S(R, \kk)_4$.  As $\mu^F$ is a chain map, we see that
\[ \dd(\mu^F(e_i \tensor e_j)) = \mu^F(\dd(e_i \tensor e_j)) = \mu^F(q_i \tensor e_j - e_i \tensor q_j) = q_ie_j - q_je_i \]
so that $\mu^F(e_i \tensor e_j)$ corresponds to a Koszul syzygy in $\syz_1^S(I)$ as wanted.  Hence, the Koszul syzygies together with multiples of the linear syzygies must span $\syz_1^S(I)_4$, and the proposition easily follows.
\end{proof}

\section{Koszul Almost Complete Intersections}
\label{Koszul:ACI's}

Recall that a standard graded $\kk$-algebra $R = S/I$ with $\hgt I = g$ is called an \emph{almost complete intersection} (or ACI for short) if $I$ is minimally generated by $g+1$ elements.

\begin{thm} \label{Koszul:algebras:with:one:linear:syzygy}
Let $R = S/I$ be a Koszul algebra with $\beta_{2,3}^S(R) = 1$.  Then there are independent linear forms $x$ and $w$ and a linear form $z$ such that $I = (xz, zw, q_3, \dots, q_{g+1})$ for some regular sequence of quadrics $q_3, \dots, q_{g+1}$ on $S/(xz,zw)$, and conversely, every ideal of this form defines a Koszul algebra with $\beta_{2,3}^S(R) = 1$.  Hence, $R$ is an almost complete intersection with $e(R) = 2^{g-1}$ and Betti table 
\textnormal{
\[
\begin{tabular}{c|cccccccc} 
  & 0 & 1 & 2 & 3 & $\cdots$ & $g-1$ & $g$ & $g+1$ \\ 
\hline 
0 & 1 & -- & -- & -- & & -- & -- & -- \\ 
1 & -- & $g+1$ & 1 & -- & & -- & -- & --  \\
2 & -- & -- & $\frac{g+2}{2}\binom{g-1}{1}$ & $\binom{g-1}{1}$ & & -- & -- & --  \\
$\vdots$ & & & & $\ddots$ & $\ddots$ \\
$g-1$ & -- & -- & -- & & & $\frac{2g-1}{g-1}\binom{g-1}{g-2}$ & $\binom{g-1}{g-2}$ & -- \\
$g$ & -- & -- & -- & & & -- & 2 & 1
\end{tabular}
\]}
Specifically, we have $\beta_{i,2i}^S(R) = \frac{g+i}{i}\binom{g-1}{i-1}$ and $\beta_{i,2i-1}^S(R) = \binom{g-1}{i-2}$ for $i \geq 2$ so that $\beta_i^S(R) = \binom{g+1}{i}$ for all $i$.
\end{thm}

\begin{proof}
Since $I$ has a linear syzygy, it is not a complete intersection.  In particular, we can write $I = (q_1, \dots, q_{g+1})$ for some linear independent quadrics $q_i$ with $g \geq 1$.  Let $U = \syz_1^S(I)$, $W \subseteq U_4$ denote the $\kk$-span of the Koszul syzygies on the $q_i$, and $\ell \in U$ denote the unique linear syzygy up to scalar multiple.  If $W \cap S_+U = 0$, then $\beta_{2,4}^S(R) \geq \binom{g+1}{2}$ so that Proposition \ref{Koszul:Betti:table:constraints} implies $I$ is a complete intersection, which is a contradiction.  Hence, there is a linear form $z$ such that $z\ell \in W$ is nonzero.  Write $z\ell = \sum_{1 \leq i < j \leq g+1} a_{i,j}(q_je_i - q_ie_j)$ for some $a_{i,j} \in \kk$, where $e_1, \dots, e_{g+1}$ denotes the standard basis of $S(-2)^{g+1}$.  After suitably relabeling the $q_i$ and rescaling the equality, we may assume that $a_{1,2} = 1$.  Reading off the first two coordinates of the preceding equality then gives $z\ell_1 = q_2 + \sum_{j = 3}^{g+1} a_{1,j}q_j$ and $z\ell_2 = -q_1 + \sum_{j = 3}^{g+1} a_{2,j}q_j$.  Using these equalities, we can replace $q_1$ and $q_2$ with $z\ell_2$ and $z\ell_1$ as generators of $I$ and assume that $q_1 = xz$ and $q_2 = zw$ for some linear forms $x$, $z$, and $w$.  Note that $x$ and $w$ must be independent since the $q_i$ are.

After making this change, we have $\ell = (w, -x, 0, \dots, 0)$ is the unique linear syzygy on the $q_i$, and $(q_2, -q_1, 0, \dots, 0) = z\ell$.  Let $W' \subseteq U_4$ denote the $\kk$-span of the Koszul syzygies other than $(q_2, -q_1, 0, \dots, 0)$.  If $W' \cap S_+U \neq 0$, then there is a linear form $v$ such that $v\ell \in W'$ is nonzero.  Write 
\[ v\ell = \sum_{\substack{1 \leq i < j \leq g+1 \\ j \geq 3}} b_{i,j}(q_je_i - q_ie_j) \] 
for some $b_{i,j} \in \kk$.  Since $b_{i,j} \neq 0$ for some $j \geq 3$, reading off the $j$-th coordinate of the above equality yields a linear dependence relation on the $q_i$, which is a contradiction.  Hence, we must have $W' \cap S_+U = 0$ so that all of the Koszul syzygies except $(q_2, -q_1, 0, \dots, 0)$ are part of a minimal set of generators for $U$.  Because $\beta_{2,4}^S(R) \leq \binom{g+1}{2} -1$ and $\beta_{2,j}^R(S) = 0$ for $j > 4$ by Lemma \ref{Koszul:algebras:have:subdiagonal:Betti:table}, it follows that $U = \syz_1^S(I)$ is minimally generated by all the Koszul syzygies on the $q_i$, except $(q_2, -q_1, 0, \dots, 0)$, together with the linear syzygy $\ell$.

If $f_{g+1} \in ((xz, zw, q_3, \dots, q_g): q_{g+1})$, then we can write $f_{g+1}q_{g+1} = -\sum_{i=1}^g f_iq_i$ for some $f_i \in S$ so that $(f_1, \dots, f_{g+1}) \in U$.  It follows from the preceding paragraph that $f_{g+1} \in (xz,zw, q_3, \dots, q_g)$ so that $q_{g+1}$ is regular on $R' = S/(xz,zw,q_3, \dots, q_g)$.  As $\reg_{R'} R = 1$, it follows from Corollary \ref{passing:Koszulness:to:and:from:quotients} that $R'$ is also Koszul.  Moreover, because we can obtain the resolution of $R$ over $S$ by taking the mapping cone of multiplication by $q_{g+1}$ on the resolution of $R'$ over $S$, it follows that $\beta^S_{2,3}(R') = 1$ and that $(w,-x,0,\dots,0)$ is the unique linear syzygy on $xz, zw, q_3, \dots, q_g$.  Hence, induction on $g$ implies that $q_3, \dots, q_{g+1}$ is a regular sequence on $S/(xz,zw)$.  Conversely, if $I = (xz, zw, q_3, \dots, q_{g+1})$ for some regular sequence of quadrics $q_3, \dots, q_{g+1}$ on $S/(xz,zw)$, it also follows from Corollary \ref{passing:Koszulness:to:and:from:quotients} that $R = S/I$ is Koszul since $S/(xz,zw)$ is Koszul.

From the preceding paragraph, we see that $\hgt I = \hgt (xz, zw) + g -1 = g$ so that $I$ is an almost complete intersection.  If $F_\bullet$ denotes the minimal free resolution of $S/(xz,zw)$ over $S$, we obtain the minimal resolution of $R$ by repeatedly taking the mapping cone of multiplication by $q_{i+1}$ on the resolution of $S/(xz,zw,q_3, \dots, q_i)$.  Since taking the mapping cone of multiplication by $q_{i+1}$ is the same as tensoring with the Koszul complex on $q_{i+1}$, we see that $F_\bullet \tensor_S K_\bullet(q_3, \dots, q_{g+1})$ is the minimal free resolution of $R$ over $S$, from which the Betti table is easily deduced.  In particular, we have
\[ \beta_i^S(R) = \frac{g+i}{i}\binom{g-1}{i-1} + \binom{g-1}{i-2}  = \binom{g}{i} + \binom{g}{i-1} = \binom{g+1}{i} \;. \]
Similarly, we note that the multiplicity of $S/(xz,zw,q_3, \dots, q_{i+1})$ is twice the multiplicity of $S/(xz,zw,q_3, \dots, q_i)$, and so, since $e(S/(xz,zw)) = 1$, we see that $e(R) = 2^{g-1}$.
\end{proof}

\begin{rmk}
In the statement of the above theorem, we can choose $x$ and $w$ so that $zw, q_3, \dots, q_{g+1}$ is a maximal $S$-regular sequence contained in $I$.  Indeed, since $q_3, \dots, q_{g+1}$ is a regular sequence on $S/(xz,zw)$, we know that $q_3, \dots, q_{g+1}$ is a regular sequence on $S$ by Auslander's Zerodivisor Theorem, which is a consequence of the Peskine-Szpiro Intersection Theorem for arbitrary Noetherian local rings and follows from results of Serre in the regular case; see \cite[II.0]{Peskine:Szpiro}.  Each associated prime of $(q_3, \dots, q_{g+1})$ cannot contain both $xz$ and $zw$, otherwise we would have $\hgt I \leq g-1$ since the former ideal is unmixed.  If either of $xz$ or $zw$ is not contained in every associated prime of $(q_3, \dots, q_{g+1})$, we are done after possibly switching the roles of $x$ and $w$.  Otherwise, if $xz$ and $zw$ are both contained in different associated primes of $(q_3, \dots, q_{g+1})$, then we can replace $w$ with $w+x$.
\end{rmk}

\begin{thm} \label{Koszul:ACI's:with:two:linear:syzygies}
Let $R = S/I$ be a Koszul almost complete intersection with $\beta_{2,3}^S(R) = 2$.  Then there is a $3 \times 2$ matrix of linear forms $M$ with $\hgt I_2(M) = 2$ such that $I = I_2(M) + (q_4, \dots, q_{g+1})$ for some regular sequence of quadrics $q_4, \dots, q_{g+1}$ on $S/I_2(M)$, and conversely, every ideal of this form defines a Koszul almost complete intersection with $\beta_{2,3}^S(R) = 2$.  Hence, $R$ has multiplicity $e(R) = 3\cdot 2^{g-2}$ and Betti table 
\textnormal{
\[
\begin{tabular}{c|cccccccc} 
  & 0 & 1 & 2 & 3 & $\cdots$ & $g-2$ & $g-1$ & $g$ \\ 
\hline 
0 & 1 & -- & -- & -- & & -- & -- & -- \\ 
1 & -- & $g+1$ & 2 & -- & & -- & -- & --  \\
2 & -- & -- & $3\binom{g-2}{1}+\binom{g-2}{2}$ & $2\binom{g-2}{1}$ & & -- & -- & --  \\
$\vdots$ & & & & $\ddots$ & $\ddots$ \\
$g-2$ & -- & -- & -- & & & $3\binom{g-2}{g-3} + 1$ & $2\binom{g-2}{g-3}$ & -- \\
$g-1$ & -- & -- & -- & & & -- & 3 & 2
\end{tabular}
\]}
Specifically, we have $\beta_{i,2i}^S(R) = 3\binom{g-2}{i-1}+ \binom{g-2}{i}$ and $\beta_{i,2i-1}^S(R) = 2\binom{g-2}{i-2}$ for $i \geq 2$ so that $\beta_i^S(R) \leq \binom{g+1}{i}$ for all $i$.
\end{thm}

\begin{proof}
Since $I$ has a linear syzygy, it is not a complete intersection.  In particular, we can write $I = (q_1, \dots, q_{g+1})$ for some linear independent quadrics $q_i$ with $g \geq 1$.  In fact, we must have $g \geq 2$ since it is easily seen that a 2-generated graded ideal cannot have two independent linear syzygies.  Let $U = \syz_1^S(I)$, $W \subseteq U_4$ denote the $\kk$-span of the Koszul syzygies on the $q_i$, and $\ell, h \in U$ denote independent linear syzygies.  Arguing as in the proof of the previous theorem, we see there are linear forms $z$ and $v$ such that $z\ell + vh \in W$ is nonzero.  Write $z\ell + vh = \sum_{1 \leq i < j \leq g+1} a_{i,j}(q_je_i - q_ie_j)$ for some $a_{i,j} \in \kk$, , where $e_1, \dots, e_{g+1}$ denotes the standard basis of $S(-2)^{g+1}$.  After suitably relabeling the $q_i$ and rescaling the equality, we may assume that $a_{1,2} = 1$.  Reading off the coordinates of the preceding equality then gives 
\begin{align*}
\tilde{q_2} = z\ell_1 + vh_1 &= q_2 + \sum_{j = 3}^{g+1} a_{1,j}q_j
\\
-\tilde{q_1} = z\ell_2 + vh_2 &= -q_1 + \sum_{j = 3}^{g+1} a_{2,j}q_j
\\
z\ell_p + vh_p &= -\sum_{i < p} a_{i,p}q_i + \sum_{i > p} a_{p,i}q_i \qquad (p \geq 3) \;.
\end{align*}
Using the above equalities, we can replace $q_1$ and $q_2$ with $\tilde{q_1}$ and $\tilde{q_2}$ as generators for $I$.  As a result, we must also replace $\ell$ with $\tilde{\ell} = (\ell_1, \ell_2,\ell_3 - a_{1,3}\ell_2 + a_{2,3}\ell_1, \dots, \ell_{g+1} - a_{1,g+1}\ell_2 + a_{2,g+1}\ell_1)$ since
\[ 0 = \sum_{j = 1}^{g+1} \ell_jq_j  = \ell_1\tilde{q_1} + \ell_2\tilde{q_2} + \sum_{j = 3}^{g+1} (\ell_j - a_{1,j}\ell_2 + a_{2,j}\ell_1)q_j \;.\]
Similarly, $h$ must be replaced with the linear syzygy $\tilde{h}$ defined as above.  It is easily seen that $\tilde{\ell}$ and $\tilde{h}$ must also be independent linear syzygies.  Finally, setting $b_{i,j} = a_{i,j} + a_{1,j}a_{2,i} - a_{1,i}a_{2,j}$ for $3 \leq i < j \leq g+1$, we claim that
\[ z\tilde{\ell} + v\tilde{h} = (\tilde{q_2}, -\tilde{q_1}, 0, \dots, 0) + \sum_{3 \leq i < j \leq g+1} b_{i,j}(q_je_i - q_ie_j) \;. \]
By definition of $\tilde{q_1}$ and $\tilde{q_2}$, it suffices to check equality in the $p$-th coordinate for $p \geq 3$.  Using the above equalities, we see that
\begin{align*} 
z\tilde{\ell}_p + v\tilde{h}_p &= z\ell_p + vh_p + a_{1,p}\tilde{q_1} + a_{2,p}\tilde{q_2}
\\
&= -\sum_{3 \leq i < p} a_{i,p}q_i + \sum_{i > p} a_{p,i}q_i - \sum_{i = 3}^{g+1} a_{1,p}a_{2,i}q_i + \sum_{i = 3}^{g+1} a_{2,p}a_{1,i}q_i 
\\
& = -\sum_{3 \leq i < p} (a_{i,p} + a_{1,p}a_{2,i} - a_{2,p}a_{1,i})q_i + \sum_{i > p} (a_{p,i} - a_{1,p}a_{2,i} + a_{2,p}a_{1,i})q_i
\\
& = -\sum_{3 \leq i < p} b_{i,p}q_i + \sum_{i > p} b_{p,i}q_i
\end{align*}
as required.  Hence, after replacing $q_1$ and $q_2$ as above, we may assume that $q_1 = -(z\ell_2 + vh_2)$, $q_2 = z\ell_1 + vh_1$, and $a_{1,j} = a_{2,j} = 0$ for all $j \geq 3$.

If $a_{i,j} \neq 0$ for some $3 \leq i < j \leq g+1$, then after relabeling the $q_i$ we may assume that $a_{3,4} \neq 0$.  Since $a_{1,j} = a_{2,j} = 0$ for all $j \geq 3$, arguing as in the preceding paragraph shows that we can replace $q_3$ and $q_4$ with $-(z\ell_4 + vh_4)$ and $z\ell_3 + vh_3$ respectively so that $I \subseteq (z, v, q_5, \dots, q_{g+1})$ has height at most $g-1$ by Krull's Height Theorem, contradicting that $I$ is an almost complete intersection.  Therefore, $a_{i,j} = 0$ for $3 \leq i < j \leq g+1$, and we see that $(q_2, -q_1, 0, \dots, 0) = z\ell + vh$ for some linear forms $z$ and $v$.

Suppose first that $z$ and $v$ are independent linear forms.  Then $z\ell_i + vh_i = 0$ for $i > 2$ implies that $(\ell_i, h_i) = a_i(v, -z)$ for some $a_i \in \kk$ so that $\ell = (\ell_1, \ell_2, a_3v, \dots, a_{g+1}v)$ and $h = (h_1, h_2, -a_3z, \dots, -a_{g+1}z)$.  If $a_i = 0$ for all $i$, then we would have two independent linear syzygies on $q_1$ and $q_2$, which we have already noted is impossible above.  Hence, after relabeling, we may assume that $a_3 \neq 0$.  Replacing $q_3$ with $a_3q_3 + \cdots + a_{g+1}q_{g+1}$, we may assume that $\ell = (\ell_1, \ell_2, v, 0, \dots, 0)$ and $h = (h_1, h_2, -z, 0, \dots, 0)$.  Therefore, we have $q_1 = -(z\ell_2 + vh_2)$ and $q_2 = z\ell_1 + vh_1$, and furthermore, $zq_3 = h_1q_1 + h_2q_2 = z(\ell_1h_2 - \ell_2h_1)$ implies $q_3 = \ell_1h_2 - \ell_2h_1$ so that $I = I_2(M) + (q_4, \dots, q_{g+1})$ where $M$ is the matrix
\begin{equation} \label{Hilbert:Burch:matrix} 
M = \begin{pmatrix} \ell_1 & h_1 \\ \ell_2 & h_2 \\ v & -z \end{pmatrix} \;.
\end{equation}

Suppose now that $v = cz$ for some $c \in \kk$.  Then after replacing $\ell$ with $\ell + ch$, we may assume that $(q_2, -q_1, 0, \dots, 0) = z\ell$ so that $q_1 = -z\ell_2$, $q_2 = z\ell_1$, and $\ell = (\ell_1, \ell_2, 0, \dots, 0)$.  Note that $\ell_1$ and $\ell_2$ must be independent linear forms or else $q_1$ and $q_2$ would not be independent.  On the other hand, we know that $\sum_{i=1}^{g+1} h_iq_i = 0$ so that $(\ell_1h_2 - \ell_2h_1)z \in (q_3, \dots, q_{g+1})$.  We claim that $z$ is a nonzerodivisor modulo $(q_3, \dots, q_{g+1})$ so that $\ell_1h_2 - \ell_2h_1 \in (q_3, \dots, q_{g+1})$.  

To see that the claim holds, we first note that $\hgt (q_3, \dots, q_{g+1}) = g-1$ so that $(q_3, \dots, q_{g+1})$ is a complete intersection.  Indeed, if this is not the case, then since $I \subseteq (z, q_3, \dots, q_{g+1})$ we would have $\hgt I \leq \hgt (q_3, \dots, q_{g+1}) + 1 \leq g-1$ by Krull's Height Theorem and \cite[III, Prop.~17]{Serre:local:algebra}, contradicting that $I$ is an almost complete intersection.  If $z$ were a zerodivisor modulo $(q_3, \dots, q_{g+1})$, then there would be an associated prime $P$ of $(q_3, \dots, q_{g+1})$ such that $I \subseteq (z, q_3, \dots, q_{g+1}) \subseteq P$ so that $\hgt I \leq g-1$ as $(q_3, \dots, q_{g+1})$ is unmixed, again contradicting that $I$ is an almost complete intersection.  And so, we see that $z$ must be a nonzerodivisor modulo $(q_3, \dots, q_{g+1})$ as claimed.  Write $\ell_1h_2 - \ell_2h_1 = a_3q_3 + \cdots + a_{g+1}q_{g+1}$ for some $a_i \in \kk$.  If $a_i = 0$ for all $i$, then $\ell_1h_2 - \ell_2h_1 = 0$ so that $(h_2, -h_1) = b(\ell_2, -\ell_1)$ for some $b \in \kk$ as $\ell_1$ and $\ell_2$ are independent linear forms.  In that case, we can replace $h$ with $h-b\ell$ and assume that $h = (0,0,h_3, \dots, h_{g+1})$ so that $q_3$ is a zerodivisor modulo $(q_4, \dots, q_{g+1})$.  However, we claim that this is impossible.  Indeed, by arguing as above, we see that $\hgt (q_4, \dots, q_{g+1}) = g-2$ so that $(q_4, \dots, q_{g+1})$ is a complete intersection, and so, if $q_3$ were a zerodivisor modulo $(q_4, \dots, q_{g+1})$, there would be an associated prime $P$ of $(q_4, \dots, q_{g+1})$ such that $(q_3, \dots, q_{g+1}) \subseteq P$ so that $\hgt (q_3, \dots, q_{g+1}) \leq g-2$ as $(q_4, \dots, q_{g+1})$ is unmixed, contradicting our earlier observation.  Hence, after relabeling, we may assume that $a_3 \neq 0$.  Replacing $q_3$ with $a_3q_3 + \cdots + a_{g+1}q_{g+1} = \ell_1h_2 - \ell_2h_1$, we see that $h = (h_1, h_2, -z, 0, \dots, 0)$ and $I = I_2(M) + (q_4, \dots, q_{g+1})$ where $M$ is the matrix of linear forms in \eqref{Hilbert:Burch:matrix} with $v = 0$.

In both of the above cases, it is easily checked that the Koszul syzygies involving any two of $q_1, q_2, q_3$ are non-minimal.  Let $W' \subseteq U_4$ denote the $\kk$-span of the other Koszul syzygies.  If $W' \cap S_+U \neq 0$, then there are linear forms $u$ and $w$ such that $u\ell + wh \in W'$ is nonzero.  Write 
\[ u\ell + wh = \sum_{j = 4}^{g+1} [b_{1,j}(q_je_1 - q_1e_j) + b_{2,j}(q_je_2 - q_2e_j)] + \sum_{3 \leq i < j \leq g+1} b_{i,j}(q_je_i - q_ie_j) \] 
for some $b_{i,j} \in \kk$.  Since $b_{i,j} \neq 0$ for some $j \geq 4$, reading off the $j$-th coordinate of the above equality yields a linear dependence relation on the $q_i$, which is a contradiction.  Hence, we must have $W' \cap S_+U = 0$ so that all of the Koszul syzygies involving at least one of $q_4, \dots, q_{g+1}$ are part of a minimal set of generators for $U$.  By Proposition \ref{Koszul:syzygies:span} and Lemma \ref{Koszul:algebras:have:subdiagonal:Betti:table}, it follows that $U = \syz_1^S(I)$ is minimally generated by all the Koszul syzygies involving at least one of $q_4, \dots, q_{g+1}$ together with the linear syzygies $\ell$ and $h$.  

If $f_{g+1} \in ((q_1, \dots, q_g): q_{g+1})$, then we can write $f_{g+1}q_{g+1} = -\sum_{i=1}^g f_iq_i$ for some $f_i \in S$ so that $(f_1, \dots, f_{g+1}) \in \syz_1^S(I)$.  It follows from the preceding paragraph that $f_{g+1} \in (q_1, \dots, q_g)$ so that $q_{g+1}$ is regular on $R' = S/(q_1, \dots, q_g)$.  As $\reg_{R'} R = 1$, it follows from Corollary \ref{passing:Koszulness:to:and:from:quotients} that $R'$ is also Koszul.  Moreover, because we can obtain the resolution of $R$ over $S$ by taking the mapping cone of multiplication by $q_{g+1}$ on the resolution of $R'$ over $S$, it follows that $\beta^S_{2,3}(R') = 2$ and that $(\ell_1,\ell_2,v,0,\dots,0)$ and $(h_1, h_2, -z, 0, \dots, 0)$ are the independent linear syzygies on $q_1,\dots, q_g$.  Hence, induction on $g$ implies that $q_4, \dots, q_{g+1}$ is a regular sequence on $S/I_2(M)$.  In particular, we see that $g = \hgt I = \hgt I_2(M) + g-2$ so that $\hgt I_2(M) = 2$.  Conversely, if $I = I_2(M) + (q_4, \dots, q_{g+1})$ for some $3 \times 2$ matrix of linear forms $M$ with $\hgt I_2(M) = 2$ and some regular sequence of quadrics $q_4, \dots, q_{g+1}$ on $S/I_2(M)$, then $S/I_2(M)$ has a Hilbert-Burch resolution by \cite[1.4.17]{Bruns:Herzog} so that $\reg_S(S/I_2(M)) = 1$, and it follows from Corollary \ref{passing:Koszulness:to:and:from:quotients} that $S/I_2(M)$, and hence also $R = S/I$, is Koszul.

If $F_\bullet$ denotes the minimal free resolution of $S/I_2(M)$ over $S$, we obtain the minimal resolution of $R$ by repeatedly taking the mapping cone of multiplication by $q_{i+1}$ on the resolution of $S/(q_1, \dots, q_i)$.  Since taking the mapping cone of multiplication by $q_{i+1}$ is the same as tensoring with the Koszul complex on $q_{i+1}$, we see that $F_\bullet \tensor_S K_\bullet(q_4, \dots, q_{g+1})$ is the minimal free resolution of $R$ over $S$, from which the Betti table is easily deduced.  In particular, we have
\[ \beta_i^S(R) = 2\binom{g-1}{i-1} + \binom{g-1}{i} = \binom{g}{i} + \binom{g-1}{i-1} \leq \binom{g}{i} + \binom{g}{i-1} = \binom{g+1}{i} \;. \]
Similarly, the multiplicity of $S/(q_1, \dots, q_{i+1})$ is twice the multiplicity of $S/(q_1, \dots, q_i)$ for $i \geq 3$, and so, since $e(S/I_2(M)) = 3$, we see that $e(R) = 3\cdot 2^{g-2}$.
\end{proof}

In the next section, we will show that every Koszul almost complete intersection has at most two linear syzygies so that the above results give a complete classification of Koszul almost complete intersections, and therefore, Question \ref{Betti:number:bound:for:Koszul:algebras} has an affirmative answer for Koszul almost complete intersections with any number of generators.  Assuming this result for the time being, we have the following corollary.

\begin{cor} \label{Koszul:ACI's:are:LG-quadratic}
Koszul almost complete intersections are LG-quadratic.
\end{cor}

\begin{proof}
Let $R = S/I$ be a Koszul almost complete intersection, and assume first that $\beta_{2,3}^S(R) = 2$ so that $I = I_2(M) + (q_4, \dots, q_{g+1})$ for some $3 \times 2$ matrix $M = (m_{ij})$ of linear forms and $q_4, \dots, q_{g+1}$ a regular sequence of quadrics on $S/I_2(M)$.  Set $\tilde{S} = S[X][y_4, \dots, y_{g+1}]$ where $X = (x_{ij})$ is a $3 \times 2$ generic matrix, $\tilde{I} = I_2(X) + (y_4^2 + q_4, \dots, y^2_{g+1} + q_{g+1})$, and $A = \tilde{S}/\tilde{I}$.  If we choose a lexicographic order on $\tilde{S}$ with $x_{1,2} > x_{1,1} > x_{2,2} > x_{2,1} > x_{3,2} > x_{3,1}$ and $y_i$ greater than the variables in $S$ for all $i$, it follows from \cite[15.15]{Eisenbud} that $\init_>(\tilde{I}) = \init_>(I_2(X)) + (y_4^2, \dots, y_{g+1}^2)$ so that the 2-minors of $X$ together with the $y_i^2+q_i$ are a Gr\"obner basis for $\tilde{I}$ by \cite{generic:ideals:of:minors:are:G-quadratic}.  Hence, $A$ is G-quadratic.  Moreover, we also know that the $y_i^2+q_i$ form a regular sequence on $\tilde{S}/I_2(X)$, and since the latter ring is Cohen-Macaulay, we see that $A$ is Cohen-Macaulay.  If $J$ denotes the ideal of $A$ generated by the linear forms $x_{ij} - m_{ij}$ and $y_s$ for $i = 1, 2, 3$, $j = 1, 2$, and $4 \leq s \leq g+1$, then $A/J \iso R$ so that $R$ will be LG-quadratic if the linear forms generating $J$ are a regular sequence.  Since $A$ is Cohen-Macaulay and $\tilde{I}$ is also an almost complete intersection of height $g$, this follows from the fact that $\hgt J = \dim A - \dim R = \dim \tilde{S} - \dim S = g+4$ is the number of generators of $J$.

Assume now that $\beta_{2,3}^S(R) = 1$ so that $I = (xz, zw, q_3, \dots, q_{g+1})$ for some linear forms $x$, $z$, and $w$ and $q_3, \dots, q_{g+1}$ a regular sequence of quadrics on $S/(xz,zw)$.  In this case, we define $S_i = S[y_{i+1}, \dots, y_{g+1}]$ and $A_i = S_i/I_i$ for $0 \leq i \leq g+1$, where
\begin{align*}
I_0 &= (y_1z, y_2z, y_3^2 + q_3, \dots, y_{g+1}^2 + q_{g+1})
\\
I_1 &= (xz, y_2z, y_3^2 + q_3, \dots, y_{g+1}^2 + q_{g+1})
\\
I_i &= (xz, zw, q_3, \dots, q_i, y_{i+1}^2 + q_{i+1}, \dots, y_{g+1}^2 + q_{g+1}) \qquad (i \geq 2) \;.
\end{align*}
As above, we see that $\init_>(I_0) = (\init_>(z)y_1, \init_>(z)y_2, y_3^2, \dots, y_{g+1}^2)$ for any monomial order on $S_0$ in which the $y_i$ are greater than every monomial in $S$ so that the generators of $I_0$ are a Gr\"obner basis and $A_0$ is G-quadratic.  In addition, we have $A_{g+1} = R$ and $A_i/(y_{i+1}) \iso A_{i+1}$ for all $i < g+1$.  An initial ideal argument as above shows that $y_{i+1}^2+q_{i+1}, \dots, y_{g+1}^2 + q_{g+1}$ is a regular sequence on $S_i/(xz, zw, q_3, \dots, q_i)$ for $i \geq 2$, and similarly, the $y_j^2+q_j$ are a regular sequence on $S_0/(y_1z,y_2z)$ and $S_1/(xz,y_1z)$.   Consequently, $A_i$ is a Koszul almost complete intersection with $\beta_{2,3}^{S_i}(A_i) = 1$ and $\hgt I_i = g$ for all $i$.  It then follows from Theorem \ref{Koszul:algebras:with:one:linear:syzygy} that the $A_i$ have the same Betti table, hence the same $h$-polynomial $h(t)$, over their respective polynomial rings $S_i$.  Hence, the Hilbert series of $A_i$ is $H_{A_i}(t) = h(t)/(1-t)^{\dim A_i}$.  We then compute that $\dim A_i = \dim S_i - g = \dim S_{i+1} -g + 1 = \dim A_{i+1} + 1$ so that $(1-t)H_{A_i}(t) = H_{A_{i+1}}(t)$ for all $i < g+1$.  This implies that the natural sequence $0 \to A_i(-1) \stackrel{y_{i+1}}{\to} A_i \to A_{i+1} \to 0$ is exact so that $y_{i+1}$ is $A_i$-regular.  Therefore, we see that $y_1, \dots, y_{g+1}$ is an $A_0$-sequence, and $R$ is LG-quadratic.
\end{proof}

\section{Linear Syzygies of Quadratic ACI's}
\label{linear:syzygies:of:quadratic:ACI's}

The following proposition is similar in spirit to Theorem \ref{Koszul:ACI's:with:two:linear:syzygies}.  However, there are two important distinctions:  We do not assume that $R$ is Koszul, so we lose some information about the syzygies of the defining ideal $I$, and to make up for this loss of information, we must assume that the ground field is infinite.  But first, we make a simple observation which will be useful in the proof.

\begin{rmk} \label{swapping:elements:of:a:regular:sequence}
If $f_1, \dots, f_n \in S$ is a regular sequence of homogeneous forms of the same degree and $f = a_1f_1 + \cdots + a_nf_n$ for some $a_i \in \kk$ with $a_1 \neq 0$, then $((f_2, \dots, f_n) : f) = ((f_2, \dots, f_n): f_1)$ so that $f, f_2, \dots, f_n$ is also a regular sequence.
\end{rmk}

\begin{prop}
Suppose that $\kk$ is an infinite field and that $R = S/I$ is an almost complete intersection defined by quadrics with $\beta_{2,3}^S(R) \geq 2$.  Then there are quadrics $q_1, \dots, q_{g+1}$ and a $3 \times 2$ matrix of linear forms $M$ such that $I = (q_1, \dots, q_{g+1})$, $q_2, \dots, q_{g+1}$ is a regular sequence, and $I_2(M) = (q_1, q_2, q_3)$.
\end{prop}

\begin{proof}
Set $g = \hgt I$.  We note that $g \geq 2$, since otherwise we would have $I = (xz, yz)$ for some linear forms $x, y, z$ so that $\beta_{2,3}^S(R) = 1$.  First, we can find quadrics $q_1, \dots, q_{g+1}$ such that $I = (q_1, \dots, q_{g+1})$ and $q_2, \dots, q_{g+1}$ is a regular sequence.  Indeed, we can take $q_{g+1}$ to be any quadric in $I$.  Having found quadrics $q_i, \dots, q_{g+1} \in I$ with $i > 2$ forming a regular sequence, we know that $I$ is not contained in any associated prime of $(q_i, \dots, q_{g+1})$ since the latter ideal is unmixed of height $g-i+2 < g$.  Because $I$ is generated in degree two, this implies $I_2 \nsubseteq P$ for each associated prime $P$ of $(q_i, \dots, q_{g+1})$.  Since $\kk$ is infinite, $I_2$ is not a union of the proper subspaces $(I \cap P)_2$ for $P \in \ass(S/(q_i, \dots, q_{g+1}))$.   Hence, we can find a quadric $q_{i-1} \in I$ so that $q_{i-1}, \dots, q_{g+1}$ is a regular sequence.  So by induction we have a regular sequence of quadrics $q_2, \dots, q_{g+1}$ in $I$, and we can take $q_1$ to be any other quadric independent from $q_2, \dots, q_{g+1}$ since $I$ is minimally generated by $g+1$ quadrics.

Let $\ell$ and $h$ be be two independent linear syzygies on the $q_i$.  Then $h_1\ell - \ell_1h$ is  a syzygy on $q_2, \dots, q_{g+1}$ and, therefore, a linear combination of Koszul syzygies.  Write
\[ h_1\ell - \ell_1h = \sum_{2 \leq i < j \leq g+1} a_{i,j}(q_je_i - q_ie_j) \]
for some $a_{i,j} \in \kk$, where $e_1, \dots, e_{g+1}$ denotes the standard basis of $S(-2)^{g+1}$.  Note that $\ell_1$ and $h_1$ must be independent linear forms, otherwise we could find a nontrivial linear syzygy on $q_2, \dots, q_{g+1}$ since $\ell$ and $h$ are independent, but that contradicts that $q_2, \dots, q_{g+1}$ is a regular sequence.  If $h_1\ell - \ell_1h = 0$, then $(-h_i, \ell_i) = b_i(h_1, -\ell_1)$ for some $b_i \in \kk$ for all $i \geq 2$ so that $h = h_1(1, -b_2, \dots, -b_{g+1})$.  But then $(1, -b_2, \dots, -b_{g+1})$ must be a syzygy on the $q_i$, contradicting that they are independent quadrics.  Hence, we see that $h_1\ell - \ell_1h \neq 0$ so that $a_{i,j} \neq 0$ for some $i,j$.  Relabeling $q_2, \dots, q_{g+1}$ if necessary, we may assume that $a_{2,3} \neq 0$.  Then by Remark \ref{swapping:elements:of:a:regular:sequence}, we can replace $q_3$ with  $q = h_1\ell_2 - \ell_1h_2 = a_{2,3}q_3 + \cdots + a_{2,g+1}q_{g+1}$.  In exchanging $q_3$ for $q$, we must replace $\ell$ with $\tilde{\ell} = (\ell_1, \ell_2, a_{2,3}^{-1}\ell_3, \ell_4 - a_{2,3}^{-1}a_{2,4}\ell_3, \dots, \ell_{g+1} - a_{2,3}^{-1}a_{2,g+1}\ell_3)$ as
\[
0 = \sum_{i=1}^{g+1} \ell_iq_i = a_{2,3}^{-1}\ell_3q + \ell_1q_1 + \ell_2q_2 + \sum_{i=4}^{g+1} (\ell_i - a_{2,3}^{-1}a_{2,i}\ell_3)q_i
\]
and we also replace $h$ with the syzygy $\tilde{h}$ defined as above.  However, $\tilde{\ell}$ and $\tilde{h}$ are still independent linear syzygies since their first coordinates are independent linear forms.  After making the above changes, we have $a_{2,3} = 1$ and $a_{2,i} = 0$ for all $i > 3$.  Then $h_1\ell_3 - \ell_1h_3 = -q_2 + a_{3,4}q_4 + \cdots + a_{3,g+1}q_{g+1}$, and we can we replace $q_2$ with $-(h_1\ell_3 - \ell_1h_3)$ as above.  In that case, we have $q_2 = -(h_1\ell_3 - \ell_1h_3)$ and $q_3 = h_1\ell_2 - \ell_1h_2$ so that
\[ 0 = \sum_{i=1}^{g+1} \ell_iq_i = \ell_1(q_1 + \ell_2h_3 -\ell_3h_2) + \sum_{i=1}^4 \ell_iq_i \]
implies that $q_1 + \ell_2h_3 -\ell_3h_2 \in ((q_4, \dots, q_{g+1}): \ell_1)$.

We claim that $\ell_1$ is a nonzerodivisor modulo $(q_4, \dots, q_{g+1})$.  If not, then $\ell_1$ is contained in an associated prime of $(q_4, \dots, q_{g+1})$ so that $\hgt (\ell_1, q_4, \dots, q_{g+1}) = g-2$.  But then $\hgt (\ell_1, h_1, q_4, \dots, q_{g+1}) \leq \hgt (\ell_1, q_4, \dots, q_{g+1}) + 1 = g-1$, and since $(q_2, \dots, q_{g+1}) \subseteq (\ell_1, h_1, q_4, \dots, q_{g+1})$, this contradicts $\hgt (q_2, \dots, q_{g+1}) = g$.  Therefore, $\ell_1$ is a nonzerodivisor modulo $(q_4, \dots, q_{g+1})$ as claimed so that $q_1 + \ell_2h_3 -\ell_3h_2 \in (q_4, \dots, q_{g+1})$.  We can then write $\ell_2h_3 - \ell_3h_2 = -q_1 + c_4q_4 + \cdots + c_{g+1}q_{g+1}$ for some $c_i \in \kk$.  Replacing $q_1$ with $\ell_2h_3 - \ell_3h_2$ and setting 
\begin{equation} \label{Hilbert:Burch:matrix:2} 
M = \begin{pmatrix} \ell_1 & h_1 \\ \ell_2 & h_2 \\ \ell_3 & h_3 \end{pmatrix}
\end{equation}
yields $I_2(M) = (q_1, q_2, q_3)$ as wanted.
\end{proof}

\begin{thm} \label{linear:syzygies:of:quadric:ACIs}
If $R = S/I$ is a quadratic almost complete intersection, then $\beta^S_{2,3}(R) \leq 2$.
\end{thm}

\begin{proof}
Suppose that $\beta^S_{2,3}(R) \geq 3$.  Let $K$ be an infinite extension field of $\kk$, and for each $\kk$-algebra $A$, set $A_K = A \tensor_\kk K$.  Then $\beta_{2,3}^{S_K}(R_K) = \beta_{2,3}^S(R)$, $\dim R_K = \dim R$, and $IS_K/I(S_K)_+ \iso I/IS_+ \tensor_S S_K \iso I/IS_+ \tensor_\kk K$ by faithfully flat base change so that $IS_K$ is still a quadratic almost complete intersection, and replacing $R$ with $R_K$, we may assume that the ground field $\kk$ is infinite.

By the preceding proposition, there are quadrics $q_1, \dots, q_{g+1}$ for $g = \hgt I \geq 2$ and a $3 \times 2$ matrix of linear forms $M$ as in \eqref{Hilbert:Burch:matrix:2} such that $I = (q_1, \dots, q_{g+1})$, $q_2, \dots, q_{g+1}$ is a regular sequence, and $I_2(M) = (q_1, q_2, q_3)$.  We may assume that $q_1, q_2, q_3$ are the minors of $M$ as in the proof of the proposition.  In that case, $\ell = (\ell_1, \ell_2, \ell_3, 0, \dots, 0)$ and $h = (h_1, h_2, h_3, 0, \dots, 0)$ are two independent linear syzygies on the $q_i$.  Let $u$ be a linear syzygy independent from $\ell$ and $h$.  By arguing as in the proof of the preceding proposition, we see that $u_1$, $\ell_1$, and $h_1$ are independent linear forms and that $u_1\ell - \ell_1u \neq 0$.  

We claim that $u_1\ell - \ell_1u$ is linear independent from $h_1\ell - \ell_1h$.  If not, then $u_1\ell - \ell_1u = c(h_1\ell - \ell_1h)$ for some nonzero $c \in \kk$.  Setting $\tilde{u} = u -ch$ and rearranging the preceding equality, we have that $\tilde{u}$ is a linear syzygy independent from $\ell$ and $h$ with $\tilde{u}_1\ell - \ell_1\tilde{u} = 0$, which is impossible as already noted in the previous paragraph.  Hence, $u_1\ell - \ell_1u$ is independent from $h_1\ell - \ell_1h$ as claimed.  In particular, we note that $g \geq 3$ since there cannot be two independent quadric syzygies on the regular sequence $q_2, q_3$.  Write
\[ u_1\ell - \ell_1u = \sum_{2 \leq i < j \leq g+1} b_{i,j}(q_je_i - q_ie_j) \]
for some $b_{i,j} \in \kk$.  Since $h_1\ell - \ell_1h = (0, q_3, -q_2, 0, \dots, 0)$ by assumption, we must have $b_{i,j} \neq 0$ for some $(i,j) \neq (2,3)$, otherwise we would have a contradiction to the claim.  In fact, by replacing $u$ with $u - b_{2,3}h$, we may assume that $b_{2,3} = 0$.

Next, we claim that there is a $j \geq 4$ such that $b_{2,j} \neq 0$ or $b_{3,j}  \neq 0$.  If not, the first three coordinates of $u_1\ell - \ell_1u$ must be zero so that $u_1\ell - \ell_1u = (0,0,0,-\ell_1u_4, \dots, -\ell_1u_{g+1})$.  But this implies that $(u_4, \dots, u_{g+1})$ is a linear syzygy on $q_4, \dots, q_{g+1}$, which is impossible since $q_4, \dots, q_{g+1}$ is a regular sequence.  Hence, relabeling if necessary, we may assume that $b_{2,4} \neq 0$ so that we can replace $q_4$ with $u_1\ell_2 - \ell_1u_2 = b_{2,4}q_4 + \cdots + b_{2,g+1}q_{g+1}$.

Finally, we claim that there is a $j \geq 5$ such that $b_{3,j} \neq 0$.  If not, then $u_1\ell_3 - \ell_1u_3 = b_{3,4}q_4 = b_{3,4}(u_1\ell_2 - \ell_1u_2)$ as $b_{2,3} = 0$.  But this implies that $(b_{3,4}u_2 -u_3, \ell_3 - b_{3,4}\ell_2) = r(u_1, -\ell_1)$ for some $r \in \kk$ as $\ell_1$ and $u_1$ are independent linear forms.  In particular, we see that $\ell_3 \in \Span\{\ell_1, \ell_2\}$ so that $I \subseteq (\ell_1, \ell_2, q_5, \dots, q_{g+1})$.  However, $\hgt (\ell_1, \ell_2, q_5, \dots, q_{g+1}) \leq g-1$, which contradicts $\hgt I = g$.  Hence, we must have $g \geq 4$, and after relabeling, we may assume that $b_{3,5} \neq 0$ and replace $q_5$ with $u_1\ell_3 - \ell_1u_3 = b_{3,4}q_4 + b_{3,5}q_5 + \cdots + b_{3,g+1}q_{g+1}$.  But then $I \subseteq (\ell_1, \ell_2, \ell_3, q_6, \dots, q_{g+1})$ and $\hgt (\ell_1, \ell_2, \ell_3, q_6, \dots, q_{g+1}) \leq g-1$, contradicting that $\hgt I = g$.  Therefore, we must have $\beta_{2,3}^S(R) \leq 2$.
\end{proof}

\section{Future Directions}

Using our main result, we have reason to believe that Question \ref{Betti:number:bound:for:Koszul:algebras} can be answered affirmatively for Koszul algebras defined by $g = 4$ quadrics.  However, Example \ref{Koszul:but:not:LG-quadratic} points to unexpected difficulties when $g = 5$.  

Our work also has unintended connections to various other conjectures of interest. For example, the Buchsbaum-Eisenbud-Horrocks Conjecture asks whether $\beta_i^S(R) \geq \binom{c}{i}$ for all $i$ if $c = \hgt I$.  This conjecture is already known to hold for quotients by monomial ideals. An even larger lower bound $\beta_i^S(R) \geq \binom{c}{i} + \binom{c-1}{i-1}$ was given for monomial ideals $I$ of finite colength that are not complete intersections by Charalambous and Evans in \cite{Charalambous}\cite{Charalambous:Evans} and for almost complete intersections directly linked to a complete intersection by Dugger in \cite[2.3]{Dugger}.  Among Koszul almost complete intersections $R = S/I$, those with one linear syzygy cannot be directly linked to a complete intersection since they are not Cohen-Macaulay.  Nonetheless, they still satisfy this larger bound.  For Koszul ACI's with two linear syzygies, at least when the ground field is infinite, the results of the previous section show that $I$ is directly linked to the complete intersection $(\ell_1, h_1, q_4, \dots, q_{g+1})$, and Theorem \ref{Koszul:ACI's:with:two:linear:syzygies} yields that Dugger's result is sharp.  A total rank version of the Charalambous-Evans bound $\sum_i \beta_i(R) \geq 2^c + 2^{c-1}$ has been studied in \cite{Charalambous:Evans:Miller}, and Boocher and Seiner have recently established this bound for all monomial ideals which are not complete intersections.  Our work affirmatively answers a couple questions they pose \cite[1.2, 1.3]{lower:bounds:for:Betti:numbers} in the Koszul ACI case.  

Another consequence of our structure theorem is that the EGH Conjecture holds for Koszul almost complete intersections in a strong form similar to \cite[2.1]{EGH:for:quadratic:monomial:ideals}.  This conjecture asks whether, given a graded ideal $I \subseteq S = \kk[x_1, \dots, x_n]$ containing a homogeneous regular sequence $f_1, \dots, f_r$ of degrees $2 \leq d_1 \leq \cdots \leq d_r$, there is a monomial ideal $J$ containing $x_1^{d_1}, \dots, x_r^{d_r}$ and having the same Hilbert function as $I$.  The EGH Conjecture is known to hold when $I$ is a quadratic monomial ideal by the preceding paper, when $I$ is a complete intersection of quadrics, or when $I$ is generated by products of linear forms \cite{EGH:for:products:of:linear:forms}, which covers quadratic ideals of height one.  To this list, we add that, if $I$ defines a Koszul almost complete intersection, we can simply take $J = (x_1^2, \dots, x_g^2, x_gx_{g+1})$ or $J = (x_1^2, \dots, x_g^2, x_{g-1}x_g)$ according to whether $I$ has one or two linear syzygies respectively.  Note for the former case that $g+1 = \pd_S R \leq n$ by Hilbert's Syzygy Theorem.  

Given the previous and new evidence for the above conjectures in the Koszul case, it is natural to ask whether they hold at the very least for Koszul algebras in general.

\section*{Acknowledgments}

The author would like to thank his advisor, Hal Schenck, for many helpful comments and discussions throughout the development of this work.  He would also like to thank Aldo Conca for pointing out a mistake in an earlier draft of this paper and the anonymous referee for helping to simplify the exposition and clarify the history of various results cited.  Evidence for the results in this paper was provided by several computations in \texttt{Macaulay2} \cite{Macaulay2}.

\end{spacing}

\end{document}